\documentclass[a4paper,11pt]{amsart}
\usepackage[T1]{fontenc}
\usepackage{lmodern}
\usepackage{amsmath,amssymb,mathrsfs,mathtools}
\usepackage[expansion=false]{microtype}
\usepackage{enumitem}
\usepackage[hidelinks]{hyperref}
\hypersetup{pdftitle={Non-Hermitian Yang-Mills connections with vanishing Chern classes},pdfauthor={Guangzhen Ren}}
\usepackage[capitalise,nameinlink,noabbrev]{cleveref}
\allowdisplaybreaks[1]
\numberwithin{equation}{section}
\theoremstyle{plain}
\newtheorem{theorem}{Theorem}[section]
\newtheorem{proposition}[theorem]{Proposition}
\newtheorem{lemma}[theorem]{Lemma}
\newtheorem{corollary}[theorem]{Corollary}
\theoremstyle{definition}

\theoremstyle{remark}
\newtheorem{remark}[theorem]{Remark}
\newcommand{\CC}{\mathbb C}
\newcommand{\RR}{\mathbb R}
\newcommand{\ZZ}{\mathbb Z}
\newcommand{\End}{\operatorname{End}}
\newcommand{\Hom}{\operatorname{Hom}}
\newcommand{\Id}{\operatorname{Id}}
\newcommand{\diag}{\operatorname{diag}}
\newcommand{\NHYM}{\operatorname{NHYM}}
\newcommand{\dbar}{\bar\partial}
\newcommand{\dd}{\mathrm d}
\newcommand{\Mgs}{\mathcal M^{gs}}
\newcommand{\Ms}{\mathcal M^{s}}
\newcommand{\cE}{\mathcal E}
\title[NHYM connections with vanishing Chern classes]{Non-Hermitian Yang--Mills connections\\with vanishing Chern classes}
\author{Guangzhen Ren}
\address{Department of Mathematics, Zhejiang International Studies University,
No.~299 Liuhe Road, Xihu District, Hangzhou 310023, Zhejiang, China}
\email{gzren@zisu.edu.cn}
\date{}

\begin{document}
\begin{abstract}
We construct non-flat non-Hermitian Yang--Mills connections with
vanishing Chern classes whose induced and conjugate-dual
holomorphic bundles are both stable,
giving a negative answer to Kaledin and Verbitsky's flatness
question even under these stability assumptions. We also give
negative answers to their questions about the \'etaleness and
bijectivity of the harmonic-class map. The examples are explicit
connections on smoothly trivial rank-two bundles over projective
K\"ahler surfaces and admit explicit positive harmonic metrics.
For the family of fixed stable bundles studied here, the
harmonic-class map is surjective. Every fibre contains exactly
one flat point, but some fibres also contain positive-dimensional
families of non-flat connections. The entire flat locus is a
smooth component mapped biholomorphically onto the target.
For a range of parameters, we determine the complete reduced analytic
germ at a distinguished nonunitary flat point. It has two smooth
branches of different dimensions whose intersection contains
two-sided stable points arbitrarily close to that point.
Finally, a single fibre of the pair of holomorphic projections
contains infinitely many flat classes and an embedded copy of
$\CC^*$ with exactly two flat members. Thus the two holomorphic
isomorphism classes together do not determine flatness.
\end{abstract}
\maketitle

\section{Introduction}\label{sec:intro}

Kaledin--Verbitsky introduced non-Hermitian Yang--Mills (NHYM)
connections to study the complex geometry of the moduli of
Hermitian--Einstein connections \cite{KV}. These complex
connections retain the curvature equations on a K\"ahler manifold
without being required to preserve a Hermitian metric. Each
determines an induced holomorphic bundle and a conjugate-dual
holomorphic bundle. Where both bundles are stable, the pair of
holomorphic projections is locally biholomorphic near the
Hermitian--Einstein locus to the product of stable-bundle moduli
with its complex conjugate \cite[Proposition~2.19]{KV}.
Thus NHYM moduli locally complexify stable-bundle moduli, viewed
as a real analytic space. Their global geometry asks how much
of this description survives away from Hermitian connections.

Vanishing Chern classes provide a natural test. For unitary
Hermitian--Einstein connections with vanishing Chern classes,
positivity in the Chern--Weil identity forces flatness \cite{Lubke}.
The trace pairing on complex curvature has no corresponding positivity. It is
therefore meaningful to ask whether topology and stability can
recover flatness, and whether holomorphic structures or harmonic
cohomology classes determine a connection globally. The questions
in \cite[Questions~8.8--8.10]{KV} concern these forms of rigidity.
We descend an elementary matrix construction to compact examples
with stability on both sides. For a family of fixed stable bundles,
we determine the full flat locus, exhibit positive-dimensional
families in selected fibres of the harmonic-class map, and compute
the reduced fixed-bundle germ at a nonunitary flat point.

On a compact K\"ahler manifold $(X,\omega)$, a complex connection
$D$ on a smooth complex vector bundle $E$ is \emph{NHYM} if
\begin{equation}\label{eq:NHYM}
 F_D\in\Omega^{1,1}(X,\End E),\qquad
 i\Lambda_\omega F_D=\lambda\Id_E
\end{equation}
for a constant $\lambda$. Taking the trace and integrating against
$\omega^{n-1}/(n-1)!$, where $n=\dim_\CC X$, shows that $c_1(E)=0$
in real cohomology forces $\lambda=0$ \cite[Definition~3.3]{SZ}.
We work with this value of $\lambda$.
The induced holomorphic structure is $D^{0,1}$; applying the same
construction to the conjugate-dual connection gives a holomorphic
structure on $\overline E^{\,*}$. Stability means slope stability
with respect to the chosen K\"ahler form. Conventions for adjoint
connections and harmonic metrics are recalled in
\cref{sec:constructions}.

The first question concerns flatness and makes no stability
assumption.

\medskip
\noindent\textbf{Question~8.8 of \cite{KV} (flatness).}
\emph{On a compact K\"ahler manifold, does the vanishing of the
Chern classes force an NHYM connection to be flat?}
\medskip

This question was restated in \cite[Question~4.4, p.~401]{SZ}.
Our surface examples give a negative answer even when both
associated holomorphic bundles are stable.
To state the construction, put
\begin{equation}\label{eq:PQ}
 P:=\begin{pmatrix}1&0\\0&-1\end{pmatrix},\qquad
 Q:=\begin{pmatrix}0&1\\1&0\end{pmatrix},
\end{equation}
and let $T_1=\CC/(\ZZ+i\ZZ)$ with its standard flat metric.
Let $p:C\to B$ be a connected unramified double cover of a compact
curve of genus $b\ge2$, with deck involution $\tau$. Choose a
nonzero holomorphic one-form $\alpha$ such that
\begin{equation}\label{eq:alpha-period}
 \tau^*\alpha=-\alpha,\qquad \int_\ell\alpha=i
\end{equation}
for a closed loop $\ell$ in $C$. Set
\[
 Z=(C\times T_1)/\langle g\rangle,\qquad
 g(c,z)=(\tau c,-z),
\]
and equip $Z$ with the descended product K\"ahler form. The fibre
lift $S=iPQ$ defines a rank-two bundle $E_Z$ and a descended
standard Hermitian metric $h_0$. The construction and the
existence of these data are detailed in \cref{sec:constructions}.

\begin{theorem}\label{thm:main}
Let $Z$, $E_Z$ and $\alpha$ be as above, let
$q:C\times T_1\to Z$ be the quotient map, and let
$s,t\in\RR\setminus\{0\}$.
The equivariant connection
\[
 \widetilde D^Z_{s,t}
 =
 \dd+sP\alpha+tQ\,\dd\bar z
\]
on the trivial rank-two bundle over $C\times T_1$ descends
to an irreducible connection $\mathcal D^Z_{s,t}$ on $E_Z$
over the projective surface $Z$.
With respect to the descended product K\"ahler form,
$\mathcal D^Z_{s,t}$ is NHYM with zero Einstein constant,
and the descended standard Hermitian metric $h_0$ is
positive definite and harmonic.
The curvature satisfies
\[
 q^*F_{\mathcal D^Z_{s,t}}
 =
 F_{\widetilde D^Z_{s,t}}
 =
 st[P,Q]\,\alpha\wedge\dd\bar z
 \not\equiv0.
\]
Moreover, $E_Z$ is smoothly trivial of rank two; in particular,
all its integral Chern classes vanish.
Both the induced and the conjugate-dual holomorphic bundles
associated with $\mathcal D^Z_{s,t}$ are stable whenever
$s,t\notin(\pi/2)\ZZ$.
\end{theorem}

\Cref{thm:main} answers Question~8.8 negatively even when
both associated holomorphic bundles are stable. This is the
additional property of the surface construction compared with
the elementary torus model.

The elementary torus model, a one-sided stable quotient surface,
and a two-sided stable example with a flat K\"ahler metric in
complex dimension three are recorded in
\cref{rem:comparison-models}.
Rank two and complex dimension two are minimal for a non-flat
example with vanishing Chern classes; see the end of
\cref{sec:constructions}.

We next fix one holomorphic structure. Let $\Ms$ denote NHYM
moduli with stable induced bundle, and let $\Mgs\subset\Ms$ be
the locus where the conjugate-dual bundle is stable as well.
Write $\Ms_0$ for stable-bundle moduli. All quotients are by
complex gauge; precise analytic conventions appear in
\cref{sec:energy}. The two holomorphic projections give
\begin{equation}\label{eq:Pi}
 \Pi=\pi\times\bar\pi:\Mgs\longrightarrow
 \Ms_0\times\overline{\Ms_0}.
\end{equation}
Let $\cE$ be stable with flat unitary Chern connection $\nabla$.
Write $\NHYM(\cE)$ for the fibre $\pi^{-1}([\cE])$.
Each class has a unique representative $\nabla+\varphi$ with the
fixed $(0,1)$ part. The \emph{harmonic-class map} of
Kaledin--Verbitsky is
\begin{equation}\label{eq:rho-definition}
 \begin{aligned}
 \rho_\cE:\NHYM(\cE)&\longrightarrow
 \mathcal H^{1,0}_{\partial_\nabla}(\End E)
 \simeq\overline{H^1(X,\End\cE)},\\
 [\nabla+\varphi]&\longmapsto H_\nabla\varphi,
 \end{aligned}
\end{equation}
where $H_\nabla$ is harmonic projection. We use the negative
Hermitian adjoint $\chi\mapsto-\chi^*$ for the displayed Hodge
identification, in accordance with the conjugate-dual Dolbeault
operator; see
\cite[(2.5)]{KV} and \cref{sec:energy}.
Let $\mathcal K_\cE$ be the Zariski closure of the local
Kuranishi image defined using harmonic coordinates, and write
$\mathcal K_\cE^{\mathrm{conj}}$ for its complex conjugate.
The next two questions concern this target and the global
behaviour of $\rho_\cE$.

\medskip
\noindent\textbf{Question~8.9 of \cite{KV} (surjectivity).}
\emph{Does every point of $\mathcal K_\cE^{\mathrm{conj}}$ occur
as $\rho_\cE([D])$ for some $[D]\in\NHYM(\cE)$?}

\medskip
\noindent\textbf{Question~8.10 of \cite{KV} (global parametrization).}
\emph{Is $\rho_\cE$ \'etale? Is it bijective?}
\medskip

We answer Question~8.9 affirmatively for the bundles constructed
here and give counterexamples to both assertions in
Question~8.10. For the bundles below,
\cref{prop:complete-flat-locus} identifies this target directly.

\begin{theorem}\label{thm:rho-main}
Fix $t\in\RR\setminus(\pi/2)\ZZ$.
For every integer $b\ge2$, let $Z$ and $E_Z$ be as in
\cref{thm:main}, and let $\cE_t$ be the stable holomorphic
structure on $E_Z$ whose Dolbeault operator is represented
upstairs by $\dbar+tQ\,\dd\bar z$.
Let $\nabla_t$ be its flat unitary Chern connection, and put
$\mathcal H_t
 :=
 \mathcal H^{1,0}_{\partial_{\nabla_t}}(\End E_Z).$
Then the following statements hold.
\begin{enumerate}[label=\textup{(\roman*)}]
\item
Under the preceding Hodge identification,
\[
 \mathcal K_{\cE_t}^{\mathrm{conj}}
 =
 \mathcal H_t,
 \qquad
 \dim_{\CC}\mathcal H_t=2b.
\]
The flat locus
$\mathcal F_t\subset\NHYM(\cE_t)$, with its reduced
analytic structure, is a smooth closed
irreducible component, and
\[
 \rho_{\cE_t}|_{\mathcal F_t}:
 \mathcal F_t\xrightarrow{\sim}\mathcal H_t
\]
is a biholomorphism.
In particular, $\mathcal F_t\simeq\CC^{2b}$ and
$\rho_{\cE_t}$ is surjective.

\item
The fibre of $\rho_{\cE_t}$ over $tQ\,\dd z$ contains an
embedded affine space
$\mathcal A_t\simeq\CC^{2b-1}$ whose intersection with
$\mathcal F_t$ consists of exactly one point.
The space $\mathcal A_t$ lies in the connected component
of $[\nabla_t]$ in $\NHYM(\cE_t)$, and
$\mathcal A_t\cap\Mgs$ is nonempty and open in $\mathcal A_t$.
Consequently, neither $\rho_{\cE_t}$ nor its restriction to
$\NHYM(\cE_t)\cap\Mgs$ is injective or \'etale.
\end{enumerate}
\end{theorem}

Thus, part~\textup{(i)} of \cref{thm:rho-main} answers Question~8.9 affirmatively for
the bundles $\cE_t$: every point of the Kuranishi target
has exactly one flat preimage under $\rho_{\cE_t}$.
Part~\textup{(ii)} answers both parts of Question~8.10
negatively. Along the contracted affine family the map is not
locally injective, so it is neither \'etale nor bijective.
These failures persist after restriction to the two-sided
stable locus.
As $b$ varies, the contracted families have unbounded
dimension $2b-1$, although the base surfaces have complex
dimension two and the bundles have rank two.

The same construction gives a complete local model.

\begin{theorem}\label{thm:local-intro}
For real $0<|t|<\pi/2$, the examples in \cref{thm:rho-main}
have a nonunitary flat connection $D_*$, represented upstairs by
$\dd+tQ\,\dd\bar z$, at which
\begin{equation}\label{eq:intro-local-germ}
 (\NHYM(\cE_t),[D_*])\simeq
 \bigl(\CC^{2b-1}\times
       \{(w,v)\in\CC^{2b-1}\times\CC:vw=0\},0\bigr)
\end{equation}
as reduced complex analytic germs. Its two smooth local
components have dimensions $4b-2$ and $2b$, and their intersection
is the singular locus. The connected component of $\NHYM(\cE_t)$ containing the
unitary Chern connection is not pure-dimensional, and the
singular locus of its reduced analytic space meets $\Mgs$.
\end{theorem}

A single fibre of $\Pi$ contains an embedded conic isomorphic to
$\CC^*$ with exactly two flat members (\cref{prop:Pi-conic}).
The same fibre contains infinitely many flat classes.

The proofs combine explicit holonomy calculations with the
analytic fixed-bundle model. An energy identity characterizes
flatness, and Fourier decomposition computes the full linearized
kernel at $D_*$. The quadratic term lies in the cokernel, so the
transverse correction vanishes and gives the complete germ
\eqref{eq:intro-local-germ}. At the unitary Chern point the
fixed-bundle germ is instead smooth and entirely flat
\eqref{eq:chern-germ}, consistent with the local flatness noted
after \cite[Question~8.8]{KV}. This verifies that local
phenomenon within the fixed-bundle fibre. The contracted families
occur away from the unitary locus, in agreement with
\cite[Propositions~2.16 and~2.19]{KV}.

The constant-coefficient torus model is a specialization of the
complex anti-self-dual construction of Dallagnol--Jardim
\cite[\S3, (26)--(27)]{DJ}. It already gives a negative answer
to Question~8.8 as literally stated, which imposes no stability
condition. Both holomorphic bundles of that model are polystable
and not stable, so it lies outside $\Ms$. Our compact descents
place non-flat examples inside $\Mgs$. The coordinate comparison and
the relation between NHYM and complex anti-self-duality are given
in \cref{rem:comparison-models}.

The nonconstant doubly periodic solutions of Mosna--Jardim
\cite{MJ} also give non-flat examples on a four-torus. Their
four-dimensional connection form descends upon quotienting the
two translation-invariant directions by a lattice. For the flat
product complex structure inducing the opposite orientation,
its self-dual curvature is primitive of type $(1,1)$.
Meazzini--Onorati study
autodual deformations and formality on hyper-K\"ahler manifolds
\cite{MO}. Pan--Shen--Zhang prove that an NHYM bundle admits
a harmonic metric exactly when its connection is semisimple,
confirming \cite[Conjecture~8.7]{KV}; see \cite[Theorem~1.2]{PSZ}.
Wu--Zhang extend this existence criterion to arbitrary connections
over compact Riemannian manifolds \cite[Theorem~1.1]{WZ}.
Our examples exhibit an explicit harmonic metric.
Wang constructs a Hermitian metric on the unobstructed
NHYM moduli locus and an almost hypercomplex structure near its
Hermitian locus \cite{Wang}. Our results concern global
fixed-bundle fibres and their reduced analytic germs, within the
compact K\"ahler setting of \cite[\S8.2]{KV}.

\Cref{sec:constructions} proves the constructions;
\cref{sec:energy} establishes the analytic model and energy
identity. \Cref{sec:rho} determines the flat component and exhibits
contracted affine families, \cref{sec:local-geometry} computes the local
germs, and \cref{sec:projections} treats the two projections.

\section{The stable surface construction}\label{sec:constructions}

Let \(D\) be a complex connection with curvature of type \((1,1)\).
Its \((0,1)\)-part defines a holomorphic structure on \(E\), while
the \((0,1)\)-part of the conjugate-dual connection defines a
holomorphic structure on \(\overline E^{\,*}\).
We call the resulting holomorphic bundles the
\emph{induced holomorphic bundle} and the
\emph{conjugate-dual holomorphic bundle}, respectively; the latter
is called the adjoint holomorphic structure in \cite[\S2.1]{KV}.
A Hermitian metric \(h\) identifies \(E\) with
\(\overline E^{\,*}\) as smooth complex vector bundles.
Under this identification, the conjugate-dual connection becomes
the \(h\)-adjoint connection \(D^\dagger\).
In a local \(h\)-unitary frame,
\[
D=\dd+A,\qquad
D^\dagger=\dd-A^*,\qquad
(D^\dagger)^{0,1}=\bar\partial-(A^{1,0})^*,
\]
where \(A^*\) denotes the conjugate transpose of \(A\), including
complex conjugation of scalar one-forms.
Thus the conjugate-dual holomorphic structure depends only on
\(D^{1,0}\).
Throughout, stability of holomorphic bundles means slope stability
with respect to the chosen K\"ahler form.
A connection is \emph{irreducible} if it preserves no nonzero
proper smooth complex subbundle of \(E\).

For a positive Hermitian metric $h$, write $D=D_h+\psi_h$,
where $D_h$ is unitary and $\psi_h$ is self-adjoint on real
tangent vectors. The metric $h$ is \emph{harmonic} if
$D_h^*\psi_h=0$ \cite[Definition~1.1]{PSZ}. The connection $D_h$ need not
be the Chern connection of $(E,D^{0,1},h)$.
The matrix notation of \eqref{eq:PQ} is completed by
\[
 R=PQ=\tfrac12[P,Q]=\begin{pmatrix}0&1\\-1&0\end{pmatrix},
 \qquad S=iR,\qquad K=[P,Q]=2R.
\]
We use the identities
\begin{equation}\label{eq:S-relations}
 \begin{gathered}
 P^*=P,\quad Q^*=Q,\quad R^2=-\Id,\\
 S^*=S=S^{-1},\quad SPS=-P,\quad SQS=-Q.
 \end{gathered}
\end{equation}
Here $A^*=\overline A^{\,T}$ denotes the Hermitian adjoint.
The flat form on $T_1$ is
$\omega_{T_1}=\frac i2\dd z\wedge\dd\bar z$.

Connected unramified double covers of $B$ exist because
$H^1(B,\ZZ/2)\ne0$. For the surface $Z$ in \cref{thm:main},
the projection
\[
 f:Z\longrightarrow B,\qquad f([c,z])=p(c),
\]
is a holomorphic fibre bundle with fibre $T_1$.
Choose a $\tau$-invariant K\"ahler form $\omega_C$ on $C$.
The pullback of the form on $Z$ is
\[
 \widetilde\omega_Z=\omega_C+\omega_{T_1},
\]
with both forms pulled back from their factors. The bundle is
\[
 E_Z=(C\times T_1\times\CC^2)/
       ((c,z,v)\sim(\tau c,-z,Sv)).
\]
Since $S$ is unitary, the standard positive Hermitian metric
descends to $h_0$ on $E_Z$. Equivariant formulas on a cover also
denote their descents.

We first recall the flat unitary criterion used to prove stability.
\begin{lemma}\label[lemma]{lem:flat-unitary}
If $\nabla$ is a flat unitary connection over a compact connected
K\"ahler manifold, then $(E,\nabla^{0,1})$ is polystable of degree
zero and is stable precisely when its holonomy is irreducible.
Every holomorphic homomorphism between flat unitary bundles is parallel.
\end{lemma}
\begin{proof}
Write $\mathcal E:=(E,\nabla^{0,1})$.
Since $\nabla$ is flat and unitary, the K\"ahler identities give
\[
 \Delta_{\partial_\nabla}
 =
 \Delta_{\dbar_\nabla}.
\]
For any holomorphic section $u$ of $\mathcal E$, we therefore have
\[
 \|\partial_\nabla u\|_{L^2}^2
 =
 \langle\Delta_{\partial_\nabla}u,u\rangle_{L^2}
 =
 \langle\Delta_{\dbar_\nabla}u,u\rangle_{L^2}
 =
 \|\dbar_\nabla u\|_{L^2}^2
 =
 0.
\]
Together with $\dbar_\nabla u=0$, this implies $\nabla u=0$.
Applying the same argument to the induced flat unitary connection
on $\operatorname{Hom}(E_1,E_2)$ shows that every holomorphic
homomorphism between flat unitary bundles is parallel.

The connection $\nabla$ is the Chern connection of $\mathcal E$
for its parallel Hermitian metric.
Flatness makes it Hermitian--Einstein with Einstein constant zero.
Hence $\mathcal E$ is polystable
\cite[Ch.~V, \S~8]{Kobayashi}, and Chern--Weil theory gives
$\deg_\omega\mathcal E=0$.

Suppose that the holonomy of $\nabla$ is irreducible.
If $\mathcal E$ were not stable, its polystable decomposition
would yield a holomorphic projection
$p:\mathcal E\to\mathcal E$ with
\[
 0<\operatorname{rk}p<\operatorname{rk}E.
\]
By the preceding argument, $p$ is parallel.
Its image is therefore a nonzero proper parallel subbundle,
contradicting irreducibility.
Thus $\mathcal E$ is stable.

Conversely, if the holonomy representation is reducible,
a nonzero proper invariant subspace determines a nonzero proper
parallel subbundle $F\subset E$.
The induced connection on $F$ is flat and unitary, so $F$ is
a holomorphic subbundle of degree zero.
Consequently,
\[
 \mu_\omega(F)=0=\mu_\omega(\mathcal E),
\]
which contradicts stability.
This proves the equivalence.
\end{proof}

All formulas on a finite cover are written in the standard frames.
Wedge products of endomorphism-valued forms include composition.
For a constant fibre lift $C_a$ of a deck transformation $a$,
the descent condition for $\dd+A$ is
\begin{equation}\label{eq:descent-criterion}
 a^*A=C_aAC_a^{-1}.
\end{equation}
The lifts below are unitary and the quotient maps are local
isometries, so both the NHYM and harmonic-metric equations descend.

\begin{proof}[Proof of \cref{thm:main}]
The Riemann--Hurwitz formula gives
$\operatorname{genus}(C)=2b-1$.
Since the $\tau$-invariant holomorphic one-forms on $C$ are
precisely the pullbacks of those on $B$, the $\pm1$-eigenspace
decomposition of $\tau^*$ on $H^0(C,K_C)$ yields
\[
 \begin{aligned}
 \dim H^0(C,K_C)^-
 &=\dim H^0(C,K_C)-\dim H^0(B,K_B)\\
 &=(2b-1)-b=b-1>0.
 \end{aligned}
\]
A nonzero holomorphic one-form has a nonzero period, since otherwise
path integration would give a global holomorphic primitive on the
compact curve. Rescaling an anti-invariant form therefore yields
\eqref{eq:alpha-period}. The involution
$g(c,z)=(\tau c,-z)$ acts freely because $\tau$ does, and it
preserves the chosen product K\"ahler metric. Thus $Z$ is smooth
and K\"ahler. It is projective because $C\times T_1$ is projective
and the quotient is finite and free.

Put $A=sP\alpha+tQ\,\dd\bar z$. Since $g^*\alpha=-\alpha$,
$g^*\dd\bar z=-\dd\bar z$ and $SPS=-P$, $SQS=-Q$,
\[
 g^*A=SAS^{-1}.
\]
The connection and $h_0$ therefore descend to $E_Z$.
To prove smooth triviality, let $\mathcal L_{\mathrm{sgn}}\to B$
be the flat line bundle defined by the sign character of the
double cover. The eigenspace
decomposition of $S$ identifies
\[
 E_Z\simeq\underline{\CC}\oplus f^*\mathcal L_{\mathrm{sgn}}.
\]
Since $\mathcal L_{\mathrm{sgn}}^{\otimes2}$ is trivial,
$2c_1(\mathcal L_{\mathrm{sgn}})=0$. The group
$H^2(B,\ZZ)\simeq\ZZ$ has no torsion, so
$c_1(\mathcal L_{\mathrm{sgn}})=0$ and the line bundle is
smoothly trivial. Hence $E_Z$ is smoothly trivial and all its
integral Chern classes vanish.

The form $\alpha$ is closed. Consequently,
\begin{equation}\label{eq:surface-curvature}
 F_{\widetilde D^Z_{s,t}}=\dd A+A\wedge A
 =st[P,Q]\,\alpha\wedge\dd\bar z
 =stK\,\alpha\wedge\dd\bar z.
\end{equation}
This is a mixed $(1,1)$-form for the product, and
$
 \Lambda_{\widetilde\omega_Z}(\alpha\wedge\dd\bar z)=0.$
It is nonzero because $st\ne0$ and $\alpha\not\equiv0$.
Thus the descended connection is non-flat and NHYM with zero
Einstein constant.

For irreducibility, write $e_\pm=(1,\pm i)^T$. Then
\[
 Ke_\pm=\pm2i e_\pm,\qquad Pe_\pm=e_\mp.
\]
On the open set where $\alpha\ne0$, a parallel line must be a
$K$-eigenline by \eqref{eq:surface-curvature}. On each connected
open subset it must therefore be one of the constant lines
$\CC e_\pm$. For a real tangent vector $X$ in the curve factor
with $\alpha(X)\ne0$,
\[
 \widetilde D^Z_{s,t,X}e_\pm
 =s\alpha(X)e_\mp\notin\CC e_\pm.
\]
Neither eigenline is parallel. A parallel line downstairs would
pull back to one upstairs, so $\mathcal D^Z_{s,t}$ is irreducible.

Write $\alpha_{\mathrm R}=\operatorname{Re}\alpha$ and
$\alpha_{\mathrm I}=\operatorname{Im}\alpha$, and $z=x+iy$.
The unitary and self-adjoint parts are
\[
 D_{h_0}=\dd+A^{\mathrm u},\qquad
 A^{\mathrm u}=isP\alpha_{\mathrm I}-itQ\,\dd y,
 \qquad\psi=sP\alpha_{\mathrm R}+tQ\,\dd x.
\]
The real and imaginary parts of a holomorphic one-form on a
curve are harmonic, and $\dd x$ is parallel. Hence $\dd^*\psi=0$.
For a product orthonormal frame $(e_j)$, all curve-factor
commutators are multiples of $[P,P]=0$, and on the elliptic
factor $A^{\mathrm u}(\partial_x)=\psi(\partial_y)=0$. Thus
\[
 D_{h_0}^*\psi
 =\dd^*\psi-\sum_j[A^{\mathrm u}(e_j),\psi(e_j)]=0.
\]
This proves harmonicity without requiring $\alpha$ to be parallel.

The splitting also illustrates the question raised by Shen--Zhang
in \cite[\S4, pp.~401--402]{SZ} about obtaining a Higgs structure
from a harmonic metric. For $\Phi=\psi^{1,0}$,
\[
 \Phi=\tfrac12(sP\alpha+tQ\,\dd z),
 \qquad
 \Phi\wedge\Phi=\tfrac{st}{4}[P,Q]\,\alpha\wedge\dd z
 \not\equiv0.
\]
For any other positive harmonic metric $h$, the comparison
formula of Pan--Shen--Zhang \cite{PSZ}, integrated over $Z$,
shows that $h_0^{-1}h$ is parallel for the induced connection
on $\End E_Z$.
By irreducibility, $h=c h_0$ for some $c>0$. Constant rescaling
preserves the adjoint connection and hence $\Phi$. Thus every
positive harmonic metric on this connection yields the same
obstruction to Higgs integrability. The filtration by Hermitian-flat
Higgs quotients in \cite[Theorem~1.1]{NZ} assumes a semistable
Higgs bundle with $c_1(E)\cdot[\omega]^{n-1}=0$ and
$\operatorname{ch}_2(E)\cdot[\omega]^{n-2}=0$.
The connection
$\mathcal D^Z_{s,t}$ itself has nonzero trace-free curvature, so
it also lies outside the projectively flat setting of
Pan--Zhang--Zhang \cite{PZZ}.

On the cover, the two holomorphic structures are represented by
the Dolbeault operators $\dbar^+$ and $\dbar^-$, with corresponding
$h_0$-Chern connections $\nabla_t^+$ and $\nabla_s^-$:
\begin{equation}\label{eq:surface-two-chern}
 \begin{aligned}
 \dbar^+&=\dbar+tQ\,\dd\bar z,&
 \nabla_t^+&=\dd+tQ(\dd\bar z-\dd z),\\
 \dbar^-&=\dbar-sP\bar\alpha,&
 \nabla_s^-&=\dd+sP(\alpha-\bar\alpha).
 \end{aligned}
\end{equation}
Both Chern connections are flat and equivariant, so they descend
to flat unitary connections on $Z$.
Choose $c_0$ on $\ell$ and regard $\ell$ as based at $c_0$.
For $M=P,Q$, denote its eigenlines by $L_\pm(M)$.
For $\nabla_t^+$, the elliptic $y$-circle has holonomy
$e^{2itQ}$, and a path in $C$ from $c_0$ to $\tau c_0$, with
elliptic coordinate $0$, gives descended holonomy $S$. Thus
\[
 e^{2itQ}|_{L_\pm(Q)}=e^{\pm2it}\Id,
 \qquad S L_\pm(Q)=L_\mp(Q).
\]
When $t\notin(\pi/2)\ZZ$, the first holonomy has precisely these
two invariant lines, while the second exchanges them. Its
holonomy is irreducible, so the induced holomorphic bundle is stable.

For $\nabla_s^-$, the normalized period gives holonomy
$H_\ell=e^{-2isP}$. If $\upsilon$ is a path from $c_0$ to
$\tau c_0$, set
\[
 I_\upsilon=\int_\upsilon(\alpha-\bar\alpha)\in i\RR,
 \qquad H_s=e^{-sI_\upsilon P}.
\]
The descended deck-loop holonomy is $SH_s$; its action and that
of $H_\ell$ satisfy
\[
 H_\ell|_{L_\pm(P)}=e^{\mp2is}\Id,
 \qquad(SH_s)L_\pm(P)=L_\mp(P).
\]
For $s\notin(\pi/2)\ZZ$, the same argument proves irreducible
unitary holonomy and hence stability of the conjugate-dual bundle.
\end{proof}

\begin{remark}\label[remark]{rem:product-polarization}
The NHYM assertion on $Z$ uses the descended product K\"ahler form.
For this form $\alpha\wedge\dd\bar z$ is primitive.
Primitivity is metric-dependent, so the NHYM assertion does not
extend automatically to every K\"ahler form on the same surface.
For $s,t\notin(\pi/2)\ZZ$, both holomorphic bundles are
defined by irreducible unitary flat connections and hence are
stable for every K\"ahler class on $Z$ by
\cref{lem:flat-unitary}.
\end{remark}
\begin{remark}\label[remark]{rem:comparison-models}
The same matrix construction also yields the following
comparison examples.
Throughout this remark, $t\in\RR\setminus\{0\}$, and all
holonomies are based at the point represented by the origin.

On a K\"ahler surface with its complex orientation, the
zero-Einstein-constant NHYM equations are equivalent to
$*F_D=-F_D$, with $*$ extended complex-linearly.
The torus model in part~\textup{(i)} specializes
\cite[\S3, (26)--(27), p.~967]{DJ}. Rename their real
coordinates $x_\mu$ as $r_\mu$ and put
\[
 (r_1,r_2,r_3,r_4)=(x^1,x^2,y^1,-y^2),\qquad
 A_1=tP,\quad A_2=tQ.
\]
The constant matrices satisfy all differential conditions in
\cite[(26)]{DJ}, while \cite[(27)]{DJ} sets
$A_3=iA_1$ and $A_4=iA_2$. Thus the connection form becomes
$tP\,\dd z^1+tQ\,\dd\bar z^2$. Their complex coordinates
$(r_1+ir_3,r_2+ir_4)$ become $(z^1,\bar z^2)$, and
$\dd r_1\wedge\dd r_2\wedge\dd r_3\wedge\dd r_4$
has the complex orientation of $(z^1,z^2)$.

\begin{enumerate}[label=\textup{(\roman*)}]
\item \emph{The torus model.}
On $T_2=T_1^2$, equipped with the flat product K\"ahler form
$\omega_{T_2}
 =
 \frac i2\sum_{j=1}^2
 \dd z^j\wedge\dd\bar z^j,$
consider
\[
 D_t=\dd+A_t,\qquad
 A_t=tP\,\dd z^1+tQ\,\dd\bar z^2.
\]
Its curvature satisfies
\[
 F_{D_t}
 =
 t^2K\,\dd z^1\wedge\dd\bar z^2,
 \qquad
 \Lambda_{\omega_{T_2}}F_{D_t}=0.
\]
For the standard Hermitian metric $h_0$, the Hermitian
decomposition $D_t=D_{t,h_0}+\psi_t$ is
\[
 D_{t,h_0}
 =
 \dd+itP\,\dd y^1-itQ\,\dd y^2,
 \qquad
 \psi_t=tP\,\dd x^1+tQ\,\dd x^2.
\]
The coefficients are constant, and the connection form of
$D_{t,h_0}$ has only $\dd y^j$-components, whereas $\psi_t$
has only $\dd x^j$-components.
Hence $D_{t,h_0}^*\psi_t=0$.
The preceding $K$-eigenline argument proves irreducibility.

The induced holomorphic bundle has flat unitary Chern
connection
\[
 \nabla_t^T
 =
 \dd+tQ(\dd\bar z^2-\dd z^2)
 =
 \dd-2itQ\,\dd y^2,
\]
and
\[
 D_t-\nabla_t^T
 =
 t(P\,\dd z^1+Q\,\dd z^2).
\]
The $Q$-eigenlines give a holomorphic splitting
$L_{t,+}\oplus L_{t,-}$ into degree-zero flat unitary
line bundles, so the induced bundle is polystable but not stable.
The conjugate-dual bundle has flat unitary Chern connection
$\dd+2itP\,\dd y^1$; its $P$-eigenlines likewise give a
splitting into degree-zero line bundles, so it too is polystable
and not stable.
The holonomies of $L_{t,\pm}$ along the $y^2$-circle are
$e^{\pm2it}$; their holonomies along the other three
coordinate circles are trivial.
By \cref{lem:flat-unitary}, holomorphic isomorphisms between
these line bundles are parallel. Consequently,
\[
 L_{t,+}\simeq L_{t,-}
 \iff t\in(\pi/2)\ZZ,
 \qquad
 L_{t,\pm}\simeq\mathcal O
 \iff t\in\pi\ZZ.
\]
The induced rank-two holomorphic bundle is therefore trivial
exactly when $t\in\pi\ZZ$:
a global holomorphic frame would be parallel by
\cref{lem:flat-unitary}.
The underlying smooth bundle is trivial for every $t$.

\item \emph{A surface with stability on one side.}
The involution
\[
 \sigma(z^1,z^2)
 =
 (z^1+\tfrac12,-z^2)
\]
acts freely on $T_2$.
With fibre lift $P$, the connection $D_t$ descends to
\[
 Y=T_2/\langle\sigma\rangle,\qquad
 E_Y
 =
 (T_2\times\CC^2)/
 \bigl((z,v)\sim(\sigma z,Pv)\bigr),
\]
since $\sigma^*A_t=PA_tP^{-1}$.
The periodic unitary matrix
\[
 G_Y(z)=\diag(1,e^{2\pi ix^1})
\]
satisfies $G_Y(\sigma z)=PG_Y(z)$ and hence gives a smooth
trivialization of $E_Y$.
With respect to the descended flat product metric, the
descended connection is irreducible and NHYM with zero
Einstein constant, and $h_0$ is harmonic.

The connection $\nabla_t^T$ also descends and is the Chern
connection of the induced holomorphic bundle.
Its holonomy contains $e^{2itQ}$, from the $y^2$-circle,
and $P$, from the loop lifting to
\[
 [0,1]\longrightarrow T_2,\qquad
 r\longmapsto(r/2,0).
\]
For $t\notin(\pi/2)\ZZ$, the matrix $e^{2itQ}$ has distinct
eigenvalues, while $P$ exchanges its two eigenlines.
Thus the holonomy is irreducible, and the induced
holomorphic bundle is stable by \cref{lem:flat-unitary}.

The conjugate-dual Chern connection pulls back to
$ \dd+2itP\,\dd y^1.$
Both this connection and the fibre lift $P$ preserve the
$P$-eigenlines.
The conjugate-dual holomorphic bundle therefore splits into
two degree-zero flat unitary line bundles and is never stable.

\item \emph{A flat K\"ahler threefold with stability on both sides.}
On $T_3=T_1^3$, the free involution
\[
 \gamma(z^1,z^2,z^3)
 =
 (-z^1,-z^2,z^3+\tfrac12)
\]
with fibre lift $S$ gives
\[
 W=T_3/\langle\gamma\rangle,\qquad
 E_W
 =
 (T_3\times\CC^2)/
 \bigl((z,v)\sim(\gamma z,Sv)\bigr).
\]
The pullback of $D_t$ from the first two factors descends,
since
\[
 \gamma^*A_t=-A_t=SA_tS^{-1}.
\]
The periodic unitary matrix
\[
 G_W(z)
 =
 \frac{I_2+S}{2}
 +
 e^{2\pi ix^3}\frac{I_2-S}{2}
\]
satisfies $G_W(\gamma z)=SG_W(z)$ and hence gives a smooth
trivialization of $E_W$.

The curvature and Hermitian decomposition are pulled back
from the torus model, with zero third-factor components.
Thus the descended connection is irreducible and non-flat,
its curvature is primitive, and $h_0$ is harmonic with
respect to the descended flat product K\"ahler metric.

Let $q_W:T_3\to W$ be the quotient map.
The flat unitary Chern connections of the induced and
conjugate-dual holomorphic bundles satisfy
\begin{equation}\label{eq:two-chern}
 q_W^*\nabla_t^{W,+}
 =
 \dd-2itQ\,\dd y^2,
 \qquad
 q_W^*\nabla_t^{W,-}
 =
 \dd+2itP\,\dd y^1.
\end{equation}
The loop lifting to
\[
 [0,1]\longrightarrow T_3,\qquad
 r\longmapsto(0,0,r/2)
\]
has holonomy $S$ for both connections.
Their holonomy groups also contain $e^{2itQ}$ and
$e^{-2itP}$, respectively.
For $t\notin(\pi/2)\ZZ$, these matrices have distinct
eigenvalues, and $S$ exchanges the corresponding eigenlines
in both cases.
Hence both holomorphic bundles are stable by
\cref{lem:flat-unitary}.
\end{enumerate}
\end{remark}

The rank and base dimension of \cref{thm:main} are minimal.
On a curve, a $(1,1)$-curvature form satisfies
$F_D=(\Lambda_\omega F_D)\omega$, so a zero-Einstein-constant
NHYM connection is flat. For rank one on a compact K\"ahler
$n$-fold, $c_1(E)=0$ in real cohomology gives $F=\dd\xi$.
The curvature is closed and primitive, and for $n\ge2$,
\[
 *F=-\frac{F\wedge\omega^{n-2}}{(n-2)!},\qquad
 \dd^*F=0,\qquad
 \|F\|_{L^2}^2=\langle\dd\xi,F\rangle
 =\langle\xi,\dd^*F\rangle=0.
\]
Thus a non-flat example requires both rank and complex dimension
at least two.

\section{The fixed-bundle analytic model}\label{sec:energy}

Fix the smooth bundle $E$ and the K\"ahler form $\omega$.
For $k>\dim_\RR X/2+1$, let $\mathscr A^k_\omega$ be the
$H^k$ solution space of \eqref{eq:NHYM} with $\lambda=0$,
and let $\mathscr G^{k+1}=H^{k+1}(X,GL(E))$ act by
complex gauge transformations. We use
\begin{equation}\label{eq:precise-moduli}
 \Ms=(\mathscr A^k_\omega)^s/\mathscr G^{k+1},\qquad
 \Mgs=\{[D]\in\Ms:(D^\dagger)^{0,1}\text{ is stable}\},
\end{equation}
where the superscript $s$ requires stability of $(E,D^{0,1})$.
The topology is the quotient Sobolev topology. The reduced complex
analytic structures are the local elliptic slice structures of
\cite[Corollaries~2.8--2.9]{KV}. Similarly, $\Ms_0$ is the
reduced moduli space of stable Dolbeault operators on $E$ modulo
complex gauge. The locus $\Mgs$ is open in $\Ms$
\cite[\S2.3, p.~13]{KV}. The maps $\pi$ and $\bar\pi$ are
those of \cite[Lemma~2.10 and \S2.3]{KV}. A map of reduced complex analytic
spaces is called \emph{\'etale} here if it is a local
biholomorphism; this convention also applies at singular points.
The fixed-bundle space $\NHYM(\cE)$ always means the reduced
analytic fibre $\pi^{-1}([\cE])$.

\subsection{Equations and analytic structure}
Let $\cE$ be stable with a flat unitary Chern connection $\nabla$.
A representative inducing its fixed Dolbeault operator has the
form $D=\nabla+\varphi$, with
$\varphi\in\Omega^{1,0}(X,\End E)$. Stability gives
\[
 \operatorname{Aut}(\cE)=\CC^*\Id,
\]
so each class in the fixed-bundle fibre has exactly one such
representative. Flatness of $\nabla$ and the K\"ahler identity give
\begin{equation}\label{eq:fixed-bundle-equations}
 \partial_\nabla\varphi+\varphi\wedge\varphi=0,\qquad
 \partial_\nabla^*\varphi=0,
 \qquad F_D=\dbar_\nabla\varphi.
\end{equation}
Conversely, these equations imply that $D$ is NHYM with zero
Einstein constant. The first equation removes its $(2,0)$
curvature, while
$\partial_\nabla^*\varphi=i\Lambda_\omega\dbar_\nabla\varphi$
removes the contraction of its $(1,1)$ curvature.
For the surface examples we need the analytic, as well as the
set-theoretic, interpretation of this representative model.
\begin{lemma}[Analytic fixed-bundle model]\label[lemma]{lem:analytic-fixed-bundle}
Let $(X,\omega)$ be a compact connected K\"ahler surface and let $\cE$ be a stable
holomorphic bundle with a flat unitary Chern connection $\nabla$.
Equip the NHYM moduli space and the fibre
$\NHYM(\cE)=\pi^{-1}([\cE])$ with their reduced complex analytic
structures. For an integer $k\ge4$, set
\[
 \mathscr Z^k_\nabla
 =\left\{\varphi\in H^k\Omega^{1,0}(\End E):
 \partial_\nabla\varphi+\varphi\wedge\varphi=0,
 \quad \partial_\nabla^*\varphi=0\right\}.
\]
Every element of $\mathscr Z^k_\nabla$ is smooth. At each solution
$\varphi_0$, this set has a finite-dimensional reduced complex
analytic germ, independent of $k$. The map
\[
 \Theta:\mathscr Z^k_\nabla\longrightarrow\NHYM(\cE),
 \qquad\varphi\longmapsto[\nabla+\varphi],
\]
is a bijection and a local biholomorphism of these germs. In particular, harmonic projection and continuous complex-linear
functionals on $H^k\Omega^{1,0}(\End E)$ restrict to holomorphic
maps on these analytic germs.
\end{lemma}
\begin{proof}
The set-theoretic model is \cite[Theorem~2.15]{KV}; we include
the identification to fix conventions before comparing analytic
structures. Every class in $\NHYM(\cE)$ has a representative whose
$(0,1)$ part is the fixed Dolbeault operator of $\cE$.
Such a representative has the form
\[
 D=\nabla+\varphi,
 \qquad
 \varphi\in H^k\Omega^{1,0}(X,\End E).
\]
Since $\nabla$ is flat,
\[
 F_D
 =
 \partial_\nabla\varphi
 +\varphi\wedge\varphi
 +\bar\partial_\nabla\varphi.
\]
Thus $F_D^{2,0}=0$ is equivalent to
\[
 \partial_\nabla\varphi+\varphi\wedge\varphi=0.
\]
The K\"ahler identity on $(1,0)$ forms gives
\[
 \partial_\nabla^*\varphi
 =
 i\Lambda_\omega\bar\partial_\nabla\varphi,
\]
so the remaining NHYM equation with zero Einstein constant
is equivalent to $\partial_\nabla^*\varphi=0$.
Two gauge-equivalent representatives with this fixed $(0,1)$ part differ
by a holomorphic automorphism of $\cE$.
Stability gives
\[
 \operatorname{Aut}(\cE)=\CC^*\Id,
\]
and these constant scalar automorphisms act trivially on
connections. Hence the representative is unique, and
$\Theta$ is a bijection.

For integers $j\ge0$, put
\begin{equation}\label{eq:sobolev-spaces}
 \begin{aligned}
 V_j
 &:=
 H^j\Omega^{1,0}(X,\End E),\\
 W_j
 &:=
 H^j\Omega^{2,0}(X,\End E)
 \oplus H^j\Omega^{0,0}(X,\End E).
 \end{aligned}
\end{equation}
Consider the map
\[
 \mathfrak F:V_k\longrightarrow W_{k-1},
 \qquad
 \mathfrak F(\varphi)
 =
 \bigl(
 \partial_\nabla\varphi+\varphi\wedge\varphi,
 \partial_\nabla^*\varphi
 \bigr).
\]
Since $\dim_\RR X=4$ and $k\ge4$, Sobolev multiplication
makes $\mathfrak F$ a holomorphic map of complex Banach
spaces.

The first-order operator
\[
 \mathsf B=(\partial_\nabla,\partial_\nabla^*):
 \Omega^{1,0}(X,\End E)
 \longrightarrow
 \Omega^{2,0}(X,\End E)
 \oplus\Omega^{0,0}(X,\End E)
\]
satisfies
$\sigma_{\mathsf B}(\xi)^*\sigma_{\mathsf B}(\xi)
 =
 |\xi^{1,0}|^2\Id$
for every real covector $\xi$.
For $\xi\ne0$, its symbol is injective.
On a complex surface, the source and target both have
complex rank $2(\operatorname{rk}E)^2$, so the symbol is
invertible and $\mathsf B$ is elliptic.
If $\varphi\in\mathscr Z^k_\nabla$, then
\[
 \mathsf B\varphi=(-\varphi\wedge\varphi,0)\in W_k.
\]
Elliptic regularity gives $\varphi\in V_{k+1}$.
Iteration proves that every solution is smooth.

Fix a solution $\varphi_0$. The linearization of
$\mathfrak F$ at $\varphi_0$ is
\[
 L\eta
 =
 \bigl(
 \partial_\nabla\eta
 +\varphi_0\wedge\eta+\eta\wedge\varphi_0,
 \partial_\nabla^*\eta
 \bigr).
\]
It has the same principal symbol as $\mathsf B$ and is therefore
a complex-linear elliptic Fredholm operator
$V_k\to W_{k-1}$.
Let
\[
 \mathsf K:=\ker L,
 \qquad
 \mathsf C:=\ker L^*,
\]
where $L^*$ denotes the formal $L^2$ adjoint.
Both spaces are finite-dimensional and consist of smooth
sections. All orthogonal complements and projections below are taken
with respect to the $L^2$ inner products induced by $\omega$
and the fixed Hermitian metric. For the spaces defined in \eqref{eq:sobolev-spaces},
elliptic theory gives
\[
 \begin{aligned}
 V_k
 &=
 \mathsf K
 \oplus
 \bigl(V_k\cap\mathsf K^{\perp_{L^2}}\bigr),\\
 W_{k-1}
 &=
 \bigl(W_{k-1}\cap\mathsf C^{\perp_{L^2}}\bigr)
 \oplus\mathsf C,
 \end{aligned}
\]
with
$
 L(V_k)
 =
 W_{k-1}\cap\mathsf C^{\perp_{L^2}}.$
The restriction of $L$ to $V_k\cap\mathsf K^{\perp_{L^2}}$
is therefore an isomorphism onto $L(V_k)$.

Let $\Pi_{\mathsf K}$ and $\Pi_{\mathsf C}$ be the
corresponding $L^2$ projections.
The holomorphic implicit function theorem applied to
\[
 (\Id-\Pi_{\mathsf C})
 \mathfrak F(\varphi_0+\zeta+u)=0,
 \qquad
 \zeta\in\mathsf K,
 \quad
 u\in V_k\cap\mathsf K^{\perp_{L^2}},
\]
gives a unique holomorphic map $u=q(\zeta)$ near zero,
with $q(0)=0$ and $dq(0)=0.$
The remaining equation is
\[
 \kappa(\zeta)
 :=
 \Pi_{\mathsf C}
 \mathfrak F\bigl(\varphi_0+\zeta+q(\zeta)\bigr)=0.
\]
Taking reduced structures therefore yields
\[
 (\mathscr Z^k_\nabla,\varphi_0)
 \simeq
 (\kappa^{-1}(0),0)_{\mathrm{red}},
\]
which is a finite-dimensional reduced complex analytic
germ.

To compare different Sobolev orders, choose the
$L^2$ generalized inverse $G$ of $L$, extended by zero
on $\mathsf C$. It satisfies
\[
 GL=\Id-\Pi_{\mathsf K},
 \qquad
 LG=\Id-\Pi_{\mathsf C},
\]
and elliptic regularity gives compatible bounded maps
\[
 G:W_m\longrightarrow V_{m+1}
 \qquad (m\ge k-1).
\]
Writing
\[
 N(\eta):=(\eta\wedge\eta,0),
\]
we have
\[
 \mathfrak F(\varphi_0+\eta)=L\eta+N(\eta).
\]
Applying $G$ to the range equation gives
\[
 q(\zeta)
 =
 -G\,N\bigl(\zeta+q(\zeta)\bigr).
\]
Since multiplication is holomorphic at every Sobolev
order $m\ge k$, this identity first shows that $q$ is
holomorphic with values in $V_{k+1}$, and then,
inductively, in every higher Sobolev space.
The same finite-dimensional spaces, projections and
generalized inverse can be used at each order.
Uniqueness in the implicit function theorem identifies
the resulting maps $q$ and $\kappa$ as germs.
Thus the natural Sobolev inclusions induce
biholomorphisms of the reduced solution germs.

We now compare this analytic structure with that of
the fixed-bundle fibre. Set
\[
 D_0=\nabla+\varphi_0.
\]
The elliptic slice construction of
\cite[Corollaries~2.8--2.9]{KV} represents the moduli germ
at $[D_0]$ by a finite-dimensional analytic slice modulo
the stabilizer of $D_0$.
This stabilizer is $\CC^*\Id$, which acts trivially
on the space of connections. Let $d_{D_0}$ be the covariant
exterior derivative induced on $\End E$, and let
$d_{D_0}^*$ be its formal $L^2$-adjoint. Keeping this adjoint
fixed, the gauge equation
\[
 d_{D_0}^*(e^{-\nu}De^\nu-D_0)=0,
 \qquad \nu\perp_{L^2}\CC\Id,
\]
has derivative $d_{D_0}^*d_{D_0}$ in $\nu$ at $(D_0,0)$.
This elliptic operator is an isomorphism from the
scalar-orthogonal part of $H^{k+1}\Omega^0(\End E)$ onto that
of $H^{k-1}\Omega^0(\End E)$, since its kernel is $\CC\Id$.
Indeed,
\[
 \ker(d_{D_0}^*d_{D_0})=\ker d_{D_0}
 \subset H^0(X,\End\cE)=\CC\Id,
 \qquad d_{D_0}\Id=0.
\]
The holomorphic implicit function theorem therefore gives a
holomorphic local normalization. To check uniqueness modulo the
full gauge group, suppose $D,D'$ are close to $D_0$ in $H^k$ and
$D'=g^{-1}Dg$. Write $g=c\Id+v$ with $v\perp_{L^2}\CC\Id$.
Then $d_{D_0}v=g(D'-D_0)-(D-D_0)g$. The elliptic estimate for
$d_{D_0}$ on this complement and Sobolev multiplication give
\[
 \|v\|_{H^{k+1}}
 \le C\bigl(\|D-D_0\|_{H^k}+\|D'-D_0\|_{H^k}\bigr)
       (|c|+\|v\|_{H^{k+1}}).
\]
For $D,D'$ sufficiently close, this forces $c\ne0$ and makes
$g/c$ close to $\Id$ in $H^{k+1}$. Its logarithm, with its
scalar part removed, gives a small normalized gauge taking $D$
to $D'$. If both are in the Coulomb slice, implicit-function
uniqueness gives $D=D'$. Thus this normalization gives the
local quotient map; changes of elliptic slice are holomorphic
by their local gauge normalizations.
Restricting the reduced slice to the reduced
fixed-bundle fibre therefore gives a reduced analytic
germ $(\Sigma,0)$ and a holomorphic $H^k$ family $D_y$,
with
\[
 D_y^{0,1}=\bar\partial_\cE+a_y,
 \qquad
 a_0=0.
\]

Write
\[
 \mathsf D_0=\bar\partial_{\End\cE},
 \qquad
 \mathsf D_yu=\mathsf D_0u+a_yu,
 \qquad
 T_y=\mathsf D_0^*\mathsf D_y,
\]
where $\mathsf D_0^*$ is the fixed adjoint. Here $\mathsf D_y$
is the Dolbeault operator on $\Hom(\cE,(E,D_y^{0,1}))$, so
$a_yu$ denotes left composition rather than a commutator.
Let $P_0$ be the $L^2$-orthogonal projection onto $\CC\Id$,
and set
\[
 \mathscr V^j
 :=
 \ker P_0\cap H^j\Omega^0(X,\End E).
\]
The operator $T_y$ is bounded from $H^{k+1}$ to $H^{k-1}$.
Moreover,
\[
 \langle T_yu,\Id\rangle_{L^2}
 =
 \langle \mathsf D_yu,\mathsf D_0\Id\rangle_{L^2}
 =
 0,
\]
so its image lies in $\mathscr V^{k-1}$.
At $y=0$,
\[
 T_0=\mathsf D_0^*\mathsf D_0,
 \qquad
 \ker T_0=H^0(X,\End\cE)=\CC\Id.
\]
Consequently,
\[
 A_0
 :=
 T_0|_{\mathscr V^{k+1}}:
 \mathscr V^{k+1}\xrightarrow{\sim}\mathscr V^{k-1}
\]
is an isomorphism.

The restrictions $A_y:=T_y|_{\mathscr V^{k+1}}$ form a
holomorphic family of bounded operators between the fixed
Sobolev spaces $\mathscr V^{k+1}$ and $\mathscr V^{k-1}$.
Since $A_0$ is an isomorphism, after shrinking $\Sigma$, each
$A_y$ is invertible and $A_y^{-1}$ depends holomorphically
on $y$.

Define
\[
 f_y:=\Id-A_y^{-1}T_y\Id.
\]
This is a holomorphic $H^{k+1}$ family satisfying
\[
 f_0=\Id,
 \qquad
 P_0f_y=\Id,
 \qquad
 T_yf_y=0.
\]
After shrinking $\Sigma$, it takes values in
$\mathscr G^{k+1}$, since $H^{k+1}$ embeds continuously
into $C^0$ and invertibility is an open condition.

For each $y\in\Sigma$, the equality
$\pi([D_y])=[\cE]$ in the Sobolev Dolbeault quotient
provides an isomorphism
\[
 w_y\in\mathscr G^{k+1},
 \qquad
 \mathsf D_yw_y=0.
\]
No holomorphic dependence on $y$ is required of this
pointwise choice.
Write
\[
 w_y=c_y\Id+v_y,
 \qquad
 v_y\in\mathscr V^{k+1}.
\]
Since $T_yw_y=0$, we obtain
\[
 A_yv_y=-c_yT_y\Id,
 \qquad
 v_y=-c_yA_y^{-1}T_y\Id.
\]
If $c_y=0$, invertibility of $A_y$ would imply $w_y=0$,
a contradiction. Hence
\[
 w_y=c_yf_y,
 \qquad
 \mathsf D_yf_y=0
\]
at every point of $\Sigma$.
The $H^k$-valued map $y\mapsto \mathsf D_yf_y$ is holomorphic.
Since $\Sigma$ is reduced, its pointwise vanishing implies
\[
 \mathsf D_yf_y=0
\]
as an identity of holomorphic families.

It follows that
\[
 f_y^{-1}D_yf_y=\nabla+\varphi_y,
 \qquad
 \varphi_y\in H^k\Omega^{1,0}(X,\End E),
\]
and $\varphi_y$ depends holomorphically on $y$.
Gauge invariance gives $\mathfrak F(\varphi_y)=0$.
The analytic reduction above and the reducedness of $\Sigma$
therefore make $y\mapsto\varphi_y$ a holomorphic map into
$(\mathscr Z^k_\nabla,\varphi_0)$.
Conversely, the holomorphic affine map
$\varphi\mapsto\nabla+\varphi$, followed by the local
analytic quotient map, induces the holomorphic map
$\Theta$ into the reduced fixed-bundle fibre.
Uniqueness of the representative with fixed Dolbeault
operator shows that these maps are inverse. Hence
$\Theta$ induces a biholomorphism of germs at every solution.

Finally, harmonic projection is bounded and complex-linear
on $V_k$. Its restriction, and the restriction of every
continuous complex-linear functional on $V_k$, is
holomorphic on the reduced analytic solution germs.
\end{proof}

Since $\partial_\nabla^2=0$, also
$(\partial_\nabla^*)^2=0$. Hodge theory identifies the degree-one
cohomology of this decreasing-degree complex with
\[
 \frac{\ker(\partial_\nabla^*\colon
       \Omega^{1,0}(\End E)\to\Omega^0(\End E))}
      {\operatorname{im}(\partial_\nabla^*\colon
       \Omega^{2,0}(\End E)\to\Omega^{1,0}(\End E))}
 \cong \mathcal H^{1,0}_{\partial_\nabla}(\End E).
\]
The negative Hermitian adjoint identifies this harmonic space
complex-linearly with $\overline{H^1(X,\End\cE)}$. Indeed, the
conjugate-dual Dolbeault operator of $\nabla+\varphi$ differs
from $\dbar_\nabla$ by $-\varphi^*$ under the metric
identification.
Thus harmonic projection records the degree-one
$\partial_\nabla^*$-cohomology class:
\[
 H_\nabla(\varphi+\partial_\nabla^*\beta)
 =H_\nabla\varphi,
 \qquad \beta\in\Omega^{2,0}(X,\End E).
\]

For compact K\"ahler surfaces, \cref{lem:analytic-fixed-bundle}
shows that harmonic projection is holomorphic on the reduced
fixed-bundle fibre. Near $[\nabla]$
it is the local Kuranishi embedding through the conjugate-dual
projection \cite[Proposition~2.16 and Corollary~2.18]{KV}.

\subsection{The energy identity}
All norms use the K\"ahler metric and the Hilbert--Schmidt
metric induced by the Hermitian metric on $E$.

\begin{proposition}[Energy identity]\label[proposition]{prop:energy}
Let $\nabla$ be a flat unitary Chern connection on a Hermitian
bundle over a compact K\"ahler manifold $(X,\omega)$. If
$D=\nabla+\varphi$, with $\varphi\in\Omega^{1,0}(X,\End E)$,
is NHYM with zero Einstein constant, then
\begin{equation}\label{eq:energy}
 \|F_D\|_{L^2}^2=\|\varphi\wedge\varphi\|_{L^2}^2.
\end{equation}
In particular, $D$ is flat if and only if
$\varphi\wedge\varphi=0$.
\end{proposition}

\begin{proof}
The fixed-bundle equations give
\[
 \partial_\nabla\varphi=-\varphi\wedge\varphi,
 \qquad \partial_\nabla^*\varphi=0,
 \qquad F_D=\dbar_\nabla\varphi.
\]
Also $\dbar_\nabla^*\varphi=0$ by type. Since the induced
unitary connection on $\End E$ is flat, the K\"ahler identities
give $\Delta_{\partial_\nabla}=\Delta_{\dbar_\nabla}$ \cite{Huybrechts}.
Consequently
\[
 \begin{aligned}
 \|F_D\|_{L^2}^2
 &=\langle\Delta_{\dbar_\nabla}\varphi,\varphi\rangle_{L^2}\\
 &=\langle\Delta_{\partial_\nabla}\varphi,\varphi\rangle_{L^2}\\
 &=\|\partial_\nabla\varphi\|_{L^2}^2
   +\|\partial_\nabla^*\varphi\|_{L^2}^2
 =\|\varphi\wedge\varphi\|_{L^2}^2.
 \end{aligned}
\]
\end{proof}

A polystable holomorphic bundle with $c_1=c_2=0$ in real
cohomology admits a Hermitian--Einstein metric
\cite{Donaldson,UY}. Its Einstein constant is zero, and the
Chern--Weil identity makes its Chern connection flat
\cite{Lubke}. Thus the preceding identity applies after fixing
such a metric and its Dolbeault operator.

\section{The fixed-bundle map}\label{sec:rho}

Harmonic projection is a local coordinate near the Chern
connection. For our fixed stable bundle, we determine its global
flat section and construct trace-free deformation directions
that it does not detect. Throughout this section, fix
$t\in\RR\setminus(\pi/2)\ZZ$ and put
\[
 \cE_t=(E_Z,\dbar+tQ\,\dd\bar z),\qquad
 \nabla_t=\dd+tQ(\dd\bar z-\dd z).
\]
The bundle $\cE_t$ is stable and $\nabla_t$ is its flat unitary
Chern connection. All formulas are written on $C\times T_1$;
equivariance makes them formulas on $Z$.

Define
\begin{equation}\label{eq:form-spaces}
 V_\pm=\{\beta\in H^0(C,K_C):\tau^*\beta=\pm\beta\},
 \qquad \dim V_+=b,\quad \dim V_-=b-1.
\end{equation}
These dimensions were established in the construction of $Z$.

\begin{lemma}\label[lemma]{lem:harmonic-target}
Let $t\in\RR\setminus(\pi/2)\ZZ$.
For the induced connection on $\End E_Z$,
\begin{equation}\label{eq:harmonic-target}
 \mathcal H_t:=\mathcal H^{1,0}_{\partial_{\nabla_t}}
                 (\End E_Z)
 =\Id V_+\oplus QV_-\oplus\CC Q\,\dd z.
\end{equation}
In particular, $\dim_\CC\mathcal H_t=2b$. Moreover,
\begin{equation}\label{eq:coexact-two-directions}
 \begin{aligned}
 \partial_{\nabla_t}^*(R\beta\wedge\dd z)&=-4tP\beta,
     &&\beta\in V_-,\\
 \partial_{\nabla_t}^*(P\delta\wedge\dd z)&=-4tR\delta,
     &&\delta\in V_+.
 \end{aligned}
\end{equation}
\end{lemma}

\begin{proof}
Let $\widetilde E$ denote the pullback of $E_Z$ to
$C\times T_1$. On this cover, the endomorphism connection
decomposes as
\[
 \End\widetilde E
 =\langle\Id,Q\rangle\oplus\mathcal L_+\oplus\mathcal L_-,
 \qquad
 \mathcal L_\pm=\CC(P\mp R),
\]
where $\langle\Id,Q\rangle$ is trivial and $\mathcal L_\pm$
are rank-one unitary local systems pulled back from $T_1$.
They satisfy
\[
 \operatorname{ad}Q|_{\mathcal L_\pm}=\pm2\Id,
 \qquad
 \operatorname{Hol}_{y}(\mathcal L_\pm)=e^{\pm4it}\ne1.
\]
A nontrivial unitary rank-one local system on a flat torus
has no harmonic forms: the Weitzenb\"ock formula makes such
forms parallel, which is incompatible with its holonomy.
The product Laplacian is the sum of the nonnegative Laplacians
of the two factors, so these summands contribute no harmonic
forms on $C\times T_1$. The K\"ahler identities identify this
conclusion with the corresponding statement for
$\partial_{\nabla_t}$.

For the trivial summand, a harmonic $(1,0)$-form is
\[
 \Id(\eta_0+a\,\dd z)+Q(\eta_1+c\,\dd z),
 \qquad
 \eta_0,\eta_1\in H^0(C,K_C),\quad a,c\in\CC.
\]
Since conjugation by $S$ fixes $\Id$ and negates $Q$,
while the deck transformation negates $\dd z$, invariance
under the lifted deck involution is equivalent to
\[
 \tau^*\eta_0=\eta_0,\qquad
 \tau^*\eta_1=-\eta_1,\qquad a=0.
\]
Harmonic forms downstairs are exactly the invariant harmonic
forms upstairs, since the covering is a finite local isometry.
This gives \eqref{eq:harmonic-target}, and
\[
 \dim_\CC\mathcal H_t
 =\dim_\CC V_++\dim_\CC V_-+1
 =b+(b-1)+1=2b.
\]

For \eqref{eq:coexact-two-directions}, the contribution from
$C$ vanishes because holomorphic one-forms on a compact curve
are harmonic. On $T_1$, $|\dd z|^2=2$ and the $(0,1)$
connection coefficient on endomorphisms is $t\operatorname{ad}Q$.
For a constant matrix $A$ and a holomorphic one-form $\eta$ on
$C$, this gives
\[
 \partial_{\nabla_t}^*(A\eta\wedge\dd z)
 =-2t\iota_{\partial/\partial z}
        ([Q,A]\eta\wedge\dd z)
 =2t[Q,A]\eta.
\]
Now use $[Q,R]=-2P$ and $[Q,P]=-2R$.
The two input forms are invariant under the lifted deck
involution: conjugation by $S$ fixes $R$ and negates $P$,
while $\tau^*\beta=-\beta$, $\tau^*\delta=\delta$ and
the deck transformation negates $\dd z$.
Thus both identities descend to $Z$.
\end{proof}

\begin{proposition}\label[proposition]{prop:rho-slice}
Let $\beta,\gamma\in V_-$, $\delta,\varepsilon\in V_+$ and
$v\in\CC$. The formula
\begin{equation}\label{eq:rho-slice-family}
 D_{\beta,\delta,\gamma,\varepsilon,v}
 =\dd+P\beta+R\delta+Q\gamma+\Id\varepsilon
       +vQ\,\dd z+tQ\,\dd\bar z
\end{equation}
defines an equivariant connection with induced bundle $\cE_t$.
It is NHYM if and only if
\begin{equation}\label{eq:rho-slice-variety}
 v\beta=0,\qquad v\delta=0.
\end{equation}
This parameter variety embeds holomorphically in the fixed-bundle
space; in particular, different parameters give different gauge classes. For the NHYM parameters,
\begin{equation}\label{eq:rho-slice-map}
 \rho_{\cE_t}([D_{\beta,\delta,\gamma,\varepsilon,v}])
 =Q\gamma+\Id\varepsilon+(v+t)Q\,\dd z,
\end{equation}
and the connection is flat precisely when
$\beta=\delta=0$.
\end{proposition}

\begin{proof}
Equivariance follows term by term from the signs in
\eqref{eq:form-spaces} and from conjugation by $S$.
All scalar forms are closed. Any wedge of two one-forms pulled
back from $C$ vanishes, because they have type $(1,0)$ on a
curve. With
$A=P\beta+R\delta+Q\gamma+\Id\varepsilon$, we obtain
\begin{equation}\label{eq:rho-slice-curvature}
 \begin{aligned}
 F_D^{2,0}&=2v(R\beta+P\delta)\wedge\dd z,\\
 F_D^{1,1}&=2t(R\beta+P\delta)\wedge\dd\bar z,\\
 F_D^{0,2}&=0.
 \end{aligned}
\end{equation}
The mixed $(1,1)$-term is primitive for the product metric.
Linear independence of $P$ and $R$ gives
\eqref{eq:rho-slice-variety}. Since $t\ne0$, it also proves
the assertion about flatness.

The difference from $\nabla_t$ is
\[
 \varphi=P\beta+R\delta+Q\gamma+\Id\varepsilon
                         +(v+t)Q\,\dd z.
\]
The first two terms are coexact by
\cref{lem:harmonic-target}; the remaining terms are harmonic.
This proves \eqref{eq:rho-slice-map}. Finally, a gauge
transformation between two members preserves their common
$(0,1)$-operator. Stability of $\cE_t$ makes it scalar, so
the connection forms, and hence all parameters, coincide.
In the unique-representative model the family is the intersection
of the solution space with a finite-dimensional affine subspace.
Its inverse is obtained by taking linear coefficient functionals
against a fixed basis of the curve forms and matrices. This also
proves the asserted embedding, not only injectivity on points.
\end{proof}

The parameter space in \cref{prop:rho-slice} is an explicit
finite-dimensional family in the fixed-bundle moduli space.
For $0<|t|<\pi/2$,
\cref{thm:local-germ} proves that it exhausts a neighbourhood
of $D_*=\dd+tQ\,\dd\bar z$. If
\[
 \mathcal W=PV_-\oplus RV_+,
 \qquad \mathcal U=\Id V_+\oplus QV_-,
 \qquad d=\dim\mathcal W=\dim\mathcal U=2b-1,
\]
then its parameter variety is
\begin{equation}\label{eq:two-branches}
 \mathcal U\times
 \bigl((\mathcal W\times\{0\})\cup
                     (\{0\}\times\CC)\bigr).
\end{equation}
The summand $\Id V_+$ in $\mathcal U$ is the scalar
deformation factor: adding its elements changes neither the
curvature nor the NHYM equations. The contracted directions
$\mathcal W$, in contrast, are all trace-free.
The branch with $\mathcal W=0$ is flat. The other branch
contains the positive-dimensional fibres of the restriction
of $\rho$. In the coordinates
$(u,w,v)\in\mathcal U\times\mathcal W\times\CC$,
\eqref{eq:rho-slice-map} reads
\[
 \rho(u,w,v)=u+(v+t)Q\,\dd z,\qquad vw=0.
\]
Thus, within this family, the fibre over
$u+aQ\,\dd z$ is one point if $a\ne t$, and is an affine
copy of $\mathcal W$ if $a=t$.

\subsection{The flat component and the contracted fibre}

\begin{proposition}\label[proposition]{prop:complete-flat-locus}
Let $t\in\RR\setminus(\pi/2)\ZZ$. The flat locus in
$\NHYM(\cE_t)$ is exactly
\begin{equation}\label{eq:complete-flat-locus}
 \mathcal F_t=\{[\nabla_t+h]:h\in\mathcal H_t\}.
\end{equation}
It is a closed irreducible component, smooth as an analytic space
and isomorphic to $\CC^{2b}$. The Kuranishi target is $\mathcal H_t$, and the restriction
$\rho_{\cE_t}|_{\mathcal F_t}$ is a biholomorphism onto it.
In particular, $\rho_{\cE_t}$ is surjective and every
fibre of $\rho_{\cE_t}$ contains exactly one flat point.
The map
\begin{equation}\label{eq:flat-retraction}
 \mathfrak r_t([D])
 =[\nabla_t+\rho_{\cE_t}([D])]
\end{equation}
is a holomorphic retraction onto $\mathcal F_t$.
\end{proposition}
\begin{proof}
For a stable bundle $\cE$ with flat unitary Chern connection
$\nabla$, Nadel's description of the Kuranishi family in harmonic
coordinates \cite[Theorem~1.1]{Nadel} gives, under the identification
in \eqref{eq:rho-definition},
\[
 \mathcal K_{\cE}^{\mathrm{conj}}
 =\{\chi\in\mathcal H^{1,0}_{\partial_\nabla}(\End E):
              \chi\wedge\chi=0\}.
\]
Indeed, the conjugate Kuranishi germ is the germ of this quadratic
cone, whose homogeneity gives the stated Zariski closure.
Each $\chi$ in the cone is closed for both parts of $\nabla$,
so $\nabla+\chi$ is flat and has harmonic class $\chi$.
Conversely, if $D=\nabla+\varphi$ is flat in its fixed-bundle
representative, \cref{prop:energy} gives $\varphi\wedge\varphi=0$.
Then \eqref{eq:fixed-bundle-equations} gives
$\partial_\nabla\varphi=\partial_\nabla^*\varphi=0$,
so $\varphi$ equals its harmonic projection.

For $\cE_t$, the coefficients of every element of $\mathcal H_t$
lie in the commutative algebra $\langle\Id,Q\rangle$ by
\cref{lem:harmonic-target}. The cone is therefore all of
$\mathcal H_t$, and the flat locus is \eqref{eq:complete-flat-locus}.
The identities
\[
 \rho_{\cE_t}([\nabla_t+\chi])=\chi,\qquad
 \mathfrak r_t^2=\mathfrak r_t
\]
prove the assertions about the restriction, uniqueness and
retraction.

The image is closed: it is the fixed-point locus of
$\mathfrak r_t$, or equivalently the zero-curvature locus
in the fixed-bundle model. It is smooth and irreducible by
its affine parametrisation. To determine the full germ at $[\nabla_t]$, apply
\cref{lem:analytic-fixed-bundle} at $\varphi_0=0$. Its
linearization and kernel are
\[
 L=\partial_{\nabla_t}\oplus\partial_{\nabla_t}^*,
 \qquad \ker L=\mathcal H_t.
\]
For every $\chi\in\mathcal H_t$ we have
\[
 L\chi=0,\qquad \chi\wedge \chi=0.
\]
Thus $q(\chi)=0$ already solves the analytic range equation.
Uniqueness in the implicit function theorem forces $q\equiv0$,
and the remaining cokernel equation vanishes identically.
Consequently,
\begin{equation}\label{eq:chern-germ}
 (\NHYM(\cE_t),[\nabla_t])\simeq(\mathcal H_t,0).
\end{equation}
The full germ there is therefore the smooth, entirely flat
germ of $\mathcal F_t$. Let $\mathcal V$ be the irreducible component
through $[\nabla_t]$. The closed analytic subset
$\mathcal V\cap\mathcal F_t$ contains an open neighbourhood
in $\mathcal V$, so the identity principle gives
$\mathcal V\subset\mathcal F_t$.
Conversely, this intersection contains an open neighbourhood
in the irreducible space $\mathcal F_t$, giving
$\mathcal F_t\subset\mathcal V$. Thus $\mathcal F_t$ is an
irreducible component.
\end{proof}

The smoothness assertion concerns $\mathcal F_t$ itself.
The ambient fixed-bundle space is singular where this component
meets another branch, as the local model in
\cref{thm:local-germ} shows.

\begin{proof}[Proof of \cref{thm:rho-main}]
Part~\textup{(i)} follows from \cref{prop:complete-flat-locus}.
For part~\textup{(ii)}, set $\gamma=\varepsilon=v=0$ in
\cref{prop:rho-slice}. The resulting affine family
\begin{equation}\label{eq:large-contracted-family}
 \mathcal D_{\beta,\delta}
 =\dd+P\beta+R\delta+tQ\,\dd\bar z,
 \qquad(\beta,\delta)\in V_-\oplus V_+,
\end{equation}
has pairwise distinct classes, common harmonic value
$tQ\,\dd z$, and a unique flat member, at $(0,0)$.
Its dimension is $2b-1=\operatorname{genus}(C)$.

The family lies in the connected component containing
$\nabla_t$. Indeed, first vary $v$ from $-t$ to $0$ with
all curve-form parameters zero, and then vary
$(\beta,\delta)$ from $(0,0)$ to its prescribed value with
$v=0$. Equation \eqref{eq:rho-slice-variety} holds along both
segments. This proves connectedness in the fixed-bundle space.

It remains to check that contraction persists where the
conjugate-dual bundle is stable. Choose $\beta_0=\alpha$,
corresponding to $s=1\notin(\pi/2)\ZZ$ in \cref{thm:main}.
Thus the point
$(\beta,\delta)=(\beta_0,0)$ has stable conjugate-dual bundle.

Since $\Mgs$ is open in $\Ms$ \cite[\S2.3]{KV}, its
intersection with this affine family is a nonempty open subset
of $V_-\oplus V_+$ on which both holomorphic bundles are stable.

The embedded affine family is constant under $\rho$.
Every point of its open two-sided stable part has a positive-dimensional
local fibre. Hence $\rho$ is not locally injective, and therefore
not \'etale, there. The same family excludes injectivity and
bijectivity in \cite[Question~8.10]{KV}.
\end{proof}

\section{The local geometry of the fixed-bundle space}
\label{sec:local-geometry}

We now determine whether the family in \cref{prop:rho-slice}
accounts for every sufficiently small fixed-bundle deformation
of $D_*$. Computing the full linearized kernel and its quadratic
obstruction shows that the transverse analytic correction is
identically zero. The resulting equations therefore describe
the complete reduced analytic germ. This identifies the extra
directions in a contracted fibre and their intersection with
the flat branch, including points stable on both holomorphic sides.

Until the parameter extension at the end of the section, assume
\begin{equation}\label{eq:local-small-parameter}
 t\in\RR,\qquad 0<|t|<\frac\pi2,
 \qquad D_*=\dd+tQ\,\dd\bar z.
\end{equation}
Retain
\[
 \mathcal U=\Id V_+\oplus QV_-,
 \qquad \mathcal W=PV_-\oplus RV_+,
 \qquad d=\dim_\CC\mathcal U=\dim_\CC\mathcal W=2b-1.
\]
The connection $D_*$ is flat, its induced bundle is $\cE_t$,
and it corresponds to $u=w=v=0$ in
\eqref{eq:rho-slice-family}.
Its conjugate-dual bundle splits into the two flat line bundles
determined by the eigenspaces of $S$. Thus $[D_*]\notin\Mgs$,
and \cref{prop:complete-flat-locus} implies that the two-sided
stable part of the fibre over $tQ\,\dd z$ contains no flat point.
We shall locate nearby two-sided stable points on the same
singular stratum in \cref{cor:local-stable-singularities}.

\begin{theorem}\label{thm:local-germ}
The reduced complex analytic germ of $\NHYM(\cE_t)$ at $[D_*]$
is isomorphic to
\begin{equation}\label{eq:local-normal-form}
 (\mathcal U,0)\times
 \left(\{(w,v)\in\mathcal W\times\CC:vw=0\},(0,0)\right).
\end{equation}
The isomorphism is induced by the explicit representatives
\begin{equation}\label{eq:local-representatives}
 (u,w,v)\longmapsto D_*+u+w+vQ\,\dd z.
\end{equation}
In these coordinates,
\begin{equation}\label{eq:local-rho-form}
 \rho_{\cE_t}(u,w,v)=u+(t+v)Q\,\dd z.
\end{equation}
The two local irreducible components are smooth and have dimensions
$4b-2$ and $2b$. Their intersection, which is the singular locus
of this germ, is smooth of dimension $2b-1$. The Zariski tangent
space at every point of the intersection has dimension $4b-1$.
\end{theorem}

\begin{proof}
We first compute the kernel of the linearized equation, then use
an analytic implicit function theorem to show that every nearby
solution has the form \eqref{eq:local-representatives}.

Let $\partial_0$ denote the $(1,0)$-part of the flat unitary
connection $\dd$ on the bundle defined by the constant lift $S$.
It also denotes the induced operator on endomorphism-valued forms.
In the unique fixed-bundle representative, write
$D=D_*+\vartheta$, with $\vartheta\in\Omega^{1,0}(\End E_Z)$.
Since $D_*^{1,0}=\partial_0$ and
$D_*-\nabla_t=tQ\,\dd z$ is harmonic, the equations
\eqref{eq:fixed-bundle-equations} become
\begin{equation}\label{eq:local-nonlinear-equations}
 \partial_0\vartheta+\vartheta\wedge\vartheta=0,
 \qquad \partial_{\nabla_t}^*\vartheta=0.
\end{equation}
Their linearization at $D_*$ is
\begin{equation}\label{eq:local-linearization}
 L\vartheta=(\partial_0\vartheta,\partial_{\nabla_t}^*\vartheta).
\end{equation}
The second summand differs from $\partial_0^*$ by a zeroth-order
term. Hence $L$ has the surface elliptic symbol considered in
\cref{lem:analytic-fixed-bundle}. We compute its kernel on
$C\times T_1$ and then take invariant
forms. Write
\[
 \vartheta=\xi+\Phi\,\dd z,
\]
where $\xi$ is an endomorphism-valued $(1,0)$-form along $C$
and $\Phi$ is an endomorphism-valued function. The product metric
and $|\dd z|^2=2$ give
\begin{equation}\label{eq:local-kernel-product}
 \partial_C\Phi-\partial_z\xi=0,
 \qquad
 \partial_C^*\xi-2\bigl(\partial_{\bar z}\Phi+t[Q,\Phi]\bigr)=0.
\end{equation}
Decompose both coefficients into the Fourier modes
$e^{2\pi i(mx+ny)}$ on $T_1$ and the weight spaces
$\mu=0,\pm2$ of $\operatorname{ad}Q$.
For a nonzero Fourier mode, put
\[
 \kappa=\pi(n+im).
\]
On this mode, $\partial_z=\kappa$ and
$\partial_{\bar z}=-\bar\kappa$. The first equation in
\eqref{eq:local-kernel-product} gives
$\xi=\kappa^{-1}\partial_C\Phi$, and the second becomes
\begin{equation}\label{eq:local-fourier-equation}
 \Delta_{\partial_C}\Phi
 =2\kappa(-\bar\kappa+t\mu)\Phi.
\end{equation}
Taking the real part of its Hermitian $L^2(C)$ inner product
with $\Phi$, we obtain
\begin{equation}\label{eq:local-fourier-estimate}
 \|\partial_C\Phi\|^2
 =\bigl(-2\pi^2(m^2+n^2)+2\pi t\mu n\bigr)\|\Phi\|^2.
\end{equation}
The coefficient on the right is at most
\[
 -2\pi^2m^2-2\pi|n|\bigl(\pi|n|-2|t|\bigr),
\]
which is strictly negative for $(m,n)\ne(0,0)$ under
\eqref{eq:local-small-parameter}. Hence $\Phi=\xi=0$ on every
nonzero Fourier mode.

For the zero mode, $\partial_C\Phi=0$ makes $\Phi$ constant on the
compact curve $C$. Integrating the second equation of
\eqref{eq:local-kernel-product} against constant matrices yields
$[Q,\Phi]=0$, because the image of $\partial_C^*$ is orthogonal
to constants. We are left with $\partial_C^*\xi=0$.
Since $C$ has complex dimension one, $\partial_C\xi=0$
automatically. The K\"ahler identities then show that $\xi$
is harmonic and holomorphic. Conjugation by $S$ fixes $\Id,R$
and negates $P,Q$, whereas the deck transformation negates
$\dd z$. The invariant curve terms are therefore
$\mathcal U\oplus\mathcal W$. The constant coefficient $\Phi$
must commute with $Q$ and satisfy $S\Phi S=-\Phi$, which leaves
precisely $\CC Q$. We have proved
\begin{equation}\label{eq:local-kernel}
 \mathscr K:=\ker L=\mathcal U\oplus\mathcal W\oplus\CC Q\,\dd z.
\end{equation}

Apply the analytic reduction of
\cref{lem:analytic-fixed-bundle} at $D_*$. Put
\[
 \mathfrak F(\vartheta)
 =L\vartheta+(\vartheta\wedge\vartheta,0),
 \qquad
 \mathsf P=P_{\operatorname{im}L}.
\]
For $k\ge4$, the reduction supplies a unique analytic map
\[
 q:(\mathscr K,0)\longrightarrow(\mathscr K^\perp\cap H^k,0),
 \qquad q(0)=dq(0)=0,
\]
satisfying
\begin{equation}\label{eq:local-range-equation}
 \mathsf P\mathfrak F(\zeta+q(\zeta))=0.
\end{equation}
Every sufficiently small solution is of the form
$\zeta+q(\zeta)$ and satisfies the remaining cokernel equation.
We show that $q$ vanishes identically.
Indeed, write
\[
 \zeta=\Id\varepsilon+Q\gamma+P\beta+R\delta+vQ\,\dd z,
 \qquad
 \varepsilon,\delta\in V_+,\quad\gamma,\beta\in V_-.
\]
The wedge of two curve terms vanishes, and the remaining terms
satisfy
\begin{equation}\label{eq:local-quadratic-obstruction}
 \zeta\wedge\zeta=2v(R\beta+P\delta)\wedge\dd z.
\end{equation}
This is a $\partial_0$-harmonic $(2,0)$-form: upstairs its
matrix coefficients are parallel and its scalar coefficients
are holomorphic top-degree forms on the compact K\"ahler
product. It is invariant and therefore harmonic downstairs.
It follows that
\[
 L^*(\zeta\wedge\zeta,0)=\partial_0^*(\zeta\wedge\zeta)=0.
\]
Thus $\mathsf P\mathfrak F(\zeta)=0$ for every sufficiently small
$\zeta\in \mathscr K$. Uniqueness in \eqref{eq:local-range-equation} gives
$q(\zeta)=0$.

Every nearby solution of \eqref{eq:local-nonlinear-equations}
therefore belongs to $\mathscr K$, and its remaining equation is precisely
\[
 v\beta=0,\qquad v\delta=0.
\]
All these solutions are smooth. By \cref{lem:analytic-fixed-bundle}, this analytic solution germ
is the reduced fixed-bundle moduli germ. This proves
\eqref{eq:local-normal-form} and
\eqref{eq:local-representatives}.
Equation \eqref{eq:local-rho-form} follows from
\eqref{eq:rho-slice-map}.

For completeness, choose linear coordinates $w_1,\ldots,w_d$
on $\mathcal W$. In the ring of convergent power series in
$(u,w,v)$, the defining ideal is
\[
 (vw_1,\ldots,vw_d)=(v)\cap(w_1,\ldots,w_d).
\]
The two ideals on the right are prime, so their intersection
is radical. The displayed equations already define a reduced
analytic space. Its two components are
\[
 \{v=0\}\simeq\mathcal U\times\mathcal W,
 \qquad
 \{w=0\}\simeq\mathcal U\times\CC,
\]
with intersection $\{w=v=0\}\simeq\mathcal U$. Away from the
intersection a neighbourhood lies in a single smooth component.
At the intersection the defining equations have no linear
terms in $(w,v)$, so the Zariski tangent dimension is
$2d+1=4b-1$, while the local dimension is $2d=4b-2$.
This proves the remaining assertions.
\end{proof}

\begin{corollary}\label[corollary]{cor:local-stable-singularities}
Arbitrarily close to $[D_*]$, the singular locus of the
fixed-bundle space contains flat connections in $\Mgs$.
At each of these points the fixed-bundle germ has the two
smooth components described in \cref{thm:local-germ}.
Moreover, their $\rho$-fibres contain a neighbourhood of the
flat point in an affine space of dimension $2b-1$, all lying
in $\Mgs$.
\end{corollary}
\begin{proof}
Take $\gamma=s\alpha$ with real $0<|s|<\pi/2$ sufficiently small
and set
\[
 D_\gamma=\dd+Q\gamma+tQ\,\dd\bar z.
\]
The normalization \eqref{eq:alpha-period} allows such forms
arbitrarily close to zero. The connection $D_\gamma$ is flat and
corresponds to the intersection point $(u,w,v)=(Q\gamma,0,0)$
in \eqref{eq:local-normal-form}. Its induced bundle is
$\cE_t$, hence stable. Its conjugate-dual holomorphic structure has
the flat unitary Chern connection
\[
 \dd+Q\gamma-Q\bar\gamma.
\]
For a loop $\ell$ on $C$, its holonomy is
\[
 \exp\left(-2iQ\operatorname{Im}\int_\ell\gamma\right).
\]
On the normalized loop $\ell$, this is $\exp(-2isQ)$,
whose two eigenvalues are distinct because $0<|s|<\pi/2$. Parallel transport along a path from
$c_0$ to $\tau c_0$ preserves those lines before the deck
identification, and $S$ exchanges them. Thus the descended
holonomy is irreducible. By \cref{lem:flat-unitary}, the
conjugate-dual bundle is stable.

In a sufficiently small representative of the germ
\eqref{eq:local-normal-form}, the equations are independent
of $u$. Their germ at $(Q\gamma,0,0)$ is therefore the same
as at the origin. Openness of the two-sided stable locus
places a neighbourhood of this point in $\Mgs$. Finally,
fixing $u=Q\gamma$ and $v=0$ gives the affine family
\[
 \dd+Q\gamma+w+tQ\,\dd\bar z,\qquad w\in\mathcal W,
\]
with common harmonic value $Q\gamma+tQ\,\dd z$.
Its only flat member is $w=0$, by
\cref{prop:rho-slice}. The asserted neighbourhood in its fibre
now follows from openness.
\end{proof}

All singularity and dimension statements in this section concern
the fixed-bundle fibre of $\pi$, including its open part in
$\Mgs$. Determining singularities of the ambient moduli space
would also require varying the induced holomorphic structure.
The preceding argument extends beyond the standing restriction
\eqref{eq:local-small-parameter}. Put
\[
 \Sigma_C=\operatorname{Spec}
 \bigl(\Delta_{\partial_C}|_{\Omega^0(C)}\bigr),
\]
and define
\[
 \mathfrak R_C=\{0\}\cup
 \left\{\pm\left(\frac{\pi n}{2}
       +\frac{\lambda}{4\pi n}\right):
       n\in\ZZ_{\ge1},\ \lambda\in\Sigma_C\right\}.
\]
Taking $\lambda=0$ gives $(\pi/2)\ZZ\subset\mathfrak R_C$,
so $\cE_t$ is stable whenever $t\notin\mathfrak R_C$.
For $t\ne0$, taking real and imaginary parts of the inner
product of \eqref{eq:local-fourier-equation} with $\Phi$ shows
that a nonzero solution on a nonzero Fourier mode must have
$m=0$ and $\mu=\pm2$, and then
\[
 \lambda=2\pi|n|(2|t|-\pi|n|)\in\Sigma_C.
\]
Thus such a mode requires $t\in\mathfrak R_C$. This set is
locally finite: a bound $|t|\le M$ bounds both
$n\le2M/\pi$ and $\lambda\le4\pi nM-2\pi^2n^2$.
For $t\notin\mathfrak R_C$ there are no nonzero modes even
upstairs, hence none in the invariant subspace. The same kernel
and vanishing transverse correction therefore give
\eqref{eq:local-normal-form} for every real
$t\notin\mathfrak R_C$. The proof of
\cref{cor:local-stable-singularities} uses only this normal
form and the stability of $\cE_t$, so that corollary extends
to the same parameters.

\begin{proof}[Proof of \cref{thm:local-intro}]
The analytic germ and all its dimensions follow from
\cref{thm:local-germ}. The two-sided stable singular points follow from
\cref{cor:local-stable-singularities}. The flat path
\[
 \dd-(1-a)tQ\,\dd z+tQ\,\dd\bar z,\qquad 0\le a\le1,
\]
joins $\nabla_t$ to $D_*$. By \cref{thm:local-germ}, the germ
at $[D_*]$ has local irreducible components of dimensions
$4b-2$ and $2b$. Since $b\ge2$, the connected component
containing these points is not pure-dimensional.
Finally, $D_*$ has holonomy $\exp(-tQ)$ along the elliptic
$x$-circle, with eigenvalue moduli $e^t$ and $e^{-t}$.
For real nonzero $t$ it therefore preserves no positive
Hermitian metric.
\end{proof}

\section{The two holomorphic projections}\label{sec:projections}

The pair of holomorphic projections supplies the local
complexification near the Hermitian--Einstein locus. Away from
that locus, even both holomorphic isomorphism classes can remain
fixed while flatness varies. The following conic gives an entire
family exhibiting this phenomenon on the two-sided stable locus.

\begin{proposition}\label[proposition]{prop:Pi-conic}
Fix $t\in\RR\setminus(\pi/2)\ZZ$ and
$\beta=s\alpha$ with $s\in\RR\setminus(\pi/2)\ZZ$.
On the affine conic
\[
 \mathcal Q=\{(a,c)\in\CC^2:a^2+c^2=1\}\simeq\CC^*,
\]
the connections
\begin{equation}\label{eq:Pi-conic-family}
 D_{a,c}=\dd+(aP+cQ)\beta+tQ\,\dd\bar z
\end{equation}
descend to pairwise inequivalent NHYM connections on $E_Z$.
All their classes belong to $\Mgs$ and have the same image
under $\Pi$. Exactly two members are flat: $(a,c)=(0,1)$
and $(0,-1)$. The family is holomorphically embedded in
$\Mgs$, so $\Pi$ is not \'etale at any of its points.
\end{proposition}

\begin{proof}
Both $P$ and $Q$ change sign under conjugation by $S$, and
$\beta$ is anti-invariant. Hence each connection descends.
Its curvature is
\begin{equation}\label{eq:Pi-conic-curvature}
 F_{D_{a,c}}=2atR\,\beta\wedge\dd\bar z.
\end{equation}
This is of type $(1,1)$ and primitive. Since $t\ne0$ and
$\beta\ne0$, flatness is equivalent to $a=0$, giving the
two stated points of $\mathcal Q$.

All induced holomorphic structures are
$\dbar+tQ\,\dd\bar z$, so the first projection is constant
and stable. To treat the second projection, set
$$U_\zeta=\zeta\mathsf P_++\mathsf P_-,$$
where $\mathsf P_+=\tfrac12(\Id+S)$, $\mathsf P_-=\tfrac12(\Id-S)$ and $\zeta\in\CC^*.$
The matrix $U_\zeta$ is invertible and commutes with $S$,
so it is a global bundle automorphism downstairs. A direct
calculation, using $SP=-iQ$, gives
\begin{equation}\label{eq:Pi-conic-conjugacy}
 U_\zeta P U_\zeta^{-1}
 =a(\zeta)P+c(\zeta)Q,
\end{equation}
where $a(\zeta)=\frac{\zeta+\zeta^{-1}}2$ and $c(\zeta)=\frac{\zeta-\zeta^{-1}}{2i}.$
This parametrises $\mathcal Q$, with inverse $\zeta=a+ic$.
Consequently,
\[
 (D_{a(\zeta),c(\zeta)})^{1,0}
 =U_\zeta(\partial+P\beta)U_\zeta^{-1}.
\]
Taking conjugate duals gives an isomorphism of the conjugate-dual
holomorphic bundles, implemented in the metric identification
by $(U_\zeta^*)^{-1}$. Those bundles are therefore all
isomorphic to the stable conjugate-dual bundle of
$\mathcal D^Z_{s,t}$. This proves that the full conic lies
in $\Mgs$ and that $\Pi$ is constant on it.

In the unique-representative model the family is cut out by
restricting to the affine plane in \eqref{eq:Pi-conic-family}
and imposing $a^2+c^2=1$. Coefficient functionals recover $a,c$
holomorphically by \cref{lem:analytic-fixed-bundle}. Thus the
conic is embedded in the fixed-bundle fibre and hence in $\Mgs$.
The map $\Pi$ is constant along it and cannot be locally
injective at any of its points.
\end{proof}

The conjugate-dual projection $\bar\pi$ on
$\NHYM(\cE_t)\cap\Mgs$ is constant on this conic, so it fails to
be locally injective at each of its points. The harmonic-class map
of Question~8.10 has a different restriction. By
\cref{lem:harmonic-target},
\[
 \rho_{\cE_t}([D_{a,c}])=cQ\beta+tQ\,\dd z.
\]
On the conic this is a double cover of an affine line, ramified
at the two flat points.

The full $\Pi$-fibre containing the conic has infinitely many
flat points. Write $D_+=D_{0,1}$ and $z=x+iy$. For $m,n\in\ZZ$, put
$\lambda_{m,n}=\pi(n+im)$ and set
\[
 G_{m,n}=\exp\bigl(2\pi i(mx+ny)Q\bigr),\qquad
 D^{m,n}=D_++\lambda_{m,n}Q\,\dd z.
\]
The periodic gauge $G_{m,n}$ descends because
$G_{m,n}(-z)=SG_{m,n}(z)S^{-1}$. Since
$G_{m,n}^{-1}\partial G_{m,n}=\lambda_{m,n}Q\,\dd z$,
the $(1,0)$ part of $D^{m,n}$ is $G_{m,n}^{-1}D_+^{1,0}G_{m,n}$.
Thus the conjugate-dual bundle is isomorphic to that of $D_+$,
and the induced bundle is $\cE_t$. Each $D^{m,n}$ is flat since
its coefficients are closed and commute. Their harmonic classes
$Q\beta+(t+\lambda_{m,n})Q\,\dd z$ are pairwise distinct.

The flat threefold $W$ of \cref{rem:comparison-models} has
the same phenomenon. The family
\[
 \mathcal D^W_{a,c}=\dd+t(aP+cQ)\,\dd z^1+tQ\,\dd\bar z^2
 \quad (a^2+c^2=1),
\]
is equivariant for the lift $S$, and its curvature is
$2at^2R\,\dd z^1\wedge\dd\bar z^2$.
For $t\notin(\pi/2)\ZZ$, its induced bundle is fixed and
stable, while \eqref{eq:Pi-conic-conjugacy} identifies all
conjugate-dual bundles with the stable one at $(1,0)$.
Scalar automorphisms of the induced bundle imply that the
members are pairwise inequivalent. Thus one $\Pi$-fibre
contains a holomorphic conic family with exactly two flat members.

The flat members of both conics are nonunitary. On $Z$,
the elliptic $x$-circle has holonomy $\exp(-tQ)$ and hence
eigenvalues $e^t,e^{-t}$. Since $t$ is real and nonzero, these
eigenvalues are not on the unit circle. The $x^2$-circle gives
exactly the same obstruction on $W$. The unique flat member of
the contracted affine space in \cref{thm:rho-main} has this
holonomy as well. Thus flatness at these points does not place
them in the unitary locus to which the local theorem applies.

%

%
%

\end{document}